\documentclass[12pt,a4paper,reqno]{amsart}
\usepackage{amsmath, amsthm, amscd, amsfonts, amssymb}
\usepackage{comment}
\usepackage[utf8]{inputenc}
\usepackage{hyperref}
\hypersetup{
    colorlinks=true,
    filecolor=magenta,      
    urlcolor=cyan,
}

\newcommand\blfootnote[1]{%
  \begingroup
  \renewcommand\thefootnote{}\footnote{#1}%
  \addtocounter{footnote}{-1}%
  \endgroup
}

\newtheorem{theorem}{Theorem}[section]
\newtheorem{lemma}[theorem]{Lemma}
\newtheorem{proposition}[theorem]{Proposition}

\newtheorem*{claim*}{Claim}

\let\R\relax
\let\C\relax
\newcommand{\R}{\ensuremath{\mathbb{R}}}
\newcommand{\C}{\ensuremath{\mathbb{C}}}
\newcommand{\A}{\ensuremath{\mathfrak{A}}}
\newcommand{\cC}{\ensuremath{\mathcal{C}}}
\newcommand{\Q}{\ensuremath{\mathcal{Q}}}

\newcommand{\RN}[1]{%
  \textup{\uppercase\expandafter{\romannumeral#1}}%
}

\DeclareMathOperator{\tr}{tr}

\DeclareMathOperator{\spann}{span}

\title[Real Hypersurfaces with commuting $R_\xi$]{Real Hypersurfaces in $Q^m$ with commuting structure Jacobi operator}

\author[N.\ Heidari]{N. Heidari}
\author[S.\ M.\ B.\ Kashani]{S.M.B. Kashani}
\author[M.\ J.\ Vanaei]{M.J. Vanaei}

\address{Dept. of Math., University of Mohaghegh Ardabili, P.O. Box: 56199-11367, Ardabil, Iran}
\address{Dept. of Pure Math., Fculty of Math. Sciences, Tarbiat Modares University, P.O. Box: 14115-134, Tehran, Iran}
\address{Dept. of Pure Math., Fculty of Math. Sciences, Tarbiat Modares University, P.O. Box: 14115-134, Tehran, Iran}

\email{n.heidari@uma.ac.ir}
\email{kashanim@modares.ac.ir}
\email{javad.vanaei@modares.ac.ir}

\begin{document}

\blfootnote{The research is supported by Iran national Science foundation via grant no. 95012382, for the second and third authors.}

\subjclass[2010]{Primary 53C40; Secondary 53C55}
\keywords{Complex Quadric, structure Jacobi operator, Reeb curvature, complex conjugation, K\"{a}hler structure, Killing Ricci tensor, Killing shape operator}

\begin{abstract}
In this paper we study real hypersurfaces in the complex quadric space $Q^m$ whose structure Jacobi operator commutes with their structure tensor field. We show that the Reeb curvature $\alpha$ of such hypersurfaces is constant and if $\alpha$ is non-zero then the hypersurface is a tube around a totally geodesic submanifold $\C P^k \subset Q^m$, where $m=2k$. We also consider Reeb flat hypersurfaces, namely, when the Reeb curvature is zero. We show that the tube around $\C P^k \subset Q^m$ ($m=2k$), with radius $\frac{\pi}{4}$ is the only Reeb flat Hopf hypersurface with commuting Ricci tensor and also the only one with commuting shape operator. Finally, we prove that there does not exist any Reeb flat Hopf hypersurfaces with non-parallel Killing Ricci tensor.
\end{abstract}

\maketitle 

\section{Introduction}
One of the main research interests in the submanifold geometry is investigating interactions between submanifolds and some given geometric structures on the ambient space. 

Let $\bar{M}$ be a K\"{a}hler manifold with the almost complex structure $J$ and $M \subset \bar{M}$ be a real hypersurface. Then $J$ induces the orthogonal decomposition $TM = \cC \oplus \spann\{\xi\}$ where $\cC = \spann\{\xi\}^\perp$ is the maximal complex subbundle of $TM$ and the vector field $\xi$, known as the \textit{Reeb vector field}, is defined by $\xi = -JN$ for some unit normal vector field $N$ on $M$. Indeed, there is an almost contact metric structure $(\phi, \xi, \eta, g)$ induced on $M$ where $\phi$ denotes the tangential component of $J$ on $M$, $\eta$ is the dual one form of $\xi$ and $g$ indicates the induced metric tensor on $M$. A Real hypersurface whose shape operator preserves the above decomposition of $TM$ is referred to as a \textit{Hopf hypersurface}. 

In some known examples of homogeneous K\"{a}hler manifolds it has been shown that Hopf hypersurfaces are related to tubes around complex submanifolds. For example, according to Okumura \cite{Oku75} the Reeb flow of a real hypersurface $M$ in the complex projective space $\C P^m$ is isometric if and only if $M$ is an open part of a tube around a totally geodesic $\C P^k \subset \C P^m$ for some $k \in \{0, \ldots, m-1 \}$. In the case of symmetric spaces with rank greater than one, Berndt and Suh \cite{BerSuh02} proved that the Reeb flow on a real hypersurface in the complex $2$-plane Grassmannian $G_2(\C^{m+2}) = SU_{m+2}/SU_mSU_2$ is isometric if and only if the hypersurface is an open part of a tube around a $G_2(\C^{m+1}) \subset G_2(\C^{m+2})$. Another Hermitian symmetric space of rank $2$ is the complex quadric space $Q^m = SO_{m+2}/SO_mSO_2$. The complex quadric space $Q^m$ has some distinguished properties as it is the only compact parallel hypersurface in $\C P^{m+1}$ which is not totally geodesic (see \cite{NakTak76}). Moreover, it can also be described as a real Grassmann manifold of compact type with rank $2$ (\cite{KobNom96}). Accordingly, apart from the complex structure $J$, $Q^m$ admits a parallel complex line bundle $\A$ which contains an $S^1$-bundle of real structures, namely, complex conjugations on the tangent spaces of $Q^m$. For $m \geq 2$, $(Q^m, J, g)$ is a Hermitian symmetric space of rank $2$ with maximal sectional curvature $4$ (\cite{Rec95}).

In comparison to $\C P^m$ and $G_2(\C^{m+2})$, classification of real hypersurfaces with isometric Reeb flow in $Q^m$ is different. In fact, it is shown in \cite{BerSuh13} that such hypersurfaces are open part of tubes around a totally geodesic $\C P^k \in Q^m$, with $m = 2k$ (see Theorem~\ref{tubes-CPk-properties}).  

One of the notions that have commonly been used to study geometry of submanifolds is Jacobi operator (see for example \cite{ChoKi97}, \cite{JeoSuhYan09}, \cite{KiNagTak10}, \cite{LeePerSanSuh09}). Let $\nabla$ denote the Levi-Civita connection on $M$ with respect to the induced metric $g$. Then for a given unit vector field $X$ on $M$ the associated Jacobi operator is defined by $R_X = R(. , X)X$ where $R$ stands for the Riemannian curvature tensor field of $M$. The notion of Jacobi operator is related to Jacobi vector fields which are known to involve many significant geometric properties. Jacobi vector fields are solutions to the Jacobi differential equation $\nabla_{\gamma'}(\nabla_{\gamma'}Y) + R(Y,\gamma')\gamma' = 0$ along a geodesic $\gamma$ in $M$ with the velocity vector field $\gamma'$. 

For real hypersurfaces in K\"{a}hler manifolds the so-called \textit{structure Jacobi operator} $R_\xi$ is of special importance. Cho and Ki \cite{ChoKi97} classified real hypersurfaces in $\C P^m$ with commuting structure Jacobi operator i.e., $R_\xi \phi = \phi R_\xi$ under the condition that $S\xi$ is a principal curvature vector where $S$ is the shape operator of the hypersurface. They proved that such hypersurfaces are congruent to tubes around a totally geodesic $\C P^k \subset \C P^m$ with $1 \leq k \leq m-1$. In the complex $2$-plane Grassmannian $G_2(\C^{m+2})$, which has three structure tensors $\phi_i$, $1 \leq i \leq 3$ due to its quaternionic K\"{a}hler structure $\mathfrak{J}$, it has been proved that there is no real Hopf hypersurface whose structure Jacobi operator commutes with $\phi_i$, $1 \leq i \leq 3$ (see \cite{JeoSuhYan09}). In this paper we investigate real hypersurfaces with commuting structure tensor in the complex quadric space $Q^m$. We consider the general case without assuming that the hypersurface is Hopf. First, we prove that the normal vector field $N$ on $M$ is $\A$-isotropic everywhere (Proposition~\ref{N-isotropic}). Then we show that $M$ is a Hopf hypersurface (Proposition~\ref{M-Hopf}) and consequently its Reeb curvature function defined by $\alpha = g(S\xi , \xi)$ is constant (Proposition~\ref{alpha-constant}). We then consider two cases; first when $\alpha$ is non-zero, where our first classification result states 
\begin{theorem}\label{alpha-non-zero}
Let $M$ be a real hypersurface in $Q^m$ with commuting structure Jacobi operator and non-zero Reeb curvature. Then, $m$ is even, say $m=2k$, and $M$ is a tube around a $\C P^k \subset Q^m$.
\end{theorem}

Then, we suppose that the hypersurface $M$ is Reeb flat, namely $\alpha = 0$. Note that by equation \eqref{phiR-Rphi.2} a Reeb flat Hopf hypersurface has commuting structure Jacobi operator and thus we can drop the condition $R_\xi \phi = \phi R_\xi$ from our assumptions in this case. In most studies on $Q^m$ concerning the Reeb curvature function, the case of $\alpha = 0$ has been excluded. In this paper we give some characterizations for Reeb flat hypersurfaces in terms of their shape operator and Ricci tensor. 

According to \cite{BerSuh13} for real Hopf hypersurfaces in $Q^m$ the shape operator is commuting ($S\phi = \phi S$) if and only if the Reeb flow of $M$ is isomteric or, equivalently, if $M$ is an open part of a tube around a $\C P^k \subset Q^m$, with $m = 2k$. The tubes around $\C P^k \subset Q^m$ are of radius $0 < r < \pi/2$. In Proposition~\ref{lambda} we show that if a non-zero $\lambda$ is an eigenvalue of the shape operator of a Reeb flat real Hopf hypersurface then so is $\frac{1}{\lambda}$. Using that we prove
\begin{theorem}\label{commuting shape operator}
Let $M$ be a Reeb flat real Hopf hypersurface in $Q^m$. Then, the shape operator of $M$ is commuting, i.e. $S\phi = \phi S$, if and only if $m$ is even, say $m = 2k$, and $M$ is a tube of radius $\frac{\pi}{4}$ around a $\C P^k \subset Q^m$.
\end{theorem}

It follows from Proposition~\ref{not-contact} that the family of Reeb flat real Hopf hypersurfaces in $Q^m$ in some sense lies between the family of contact hypersurfaces in one side, and some of the known families of almost contact hypersurfaces such as co-symplectic and (nearly) Kenmutso hypersurfaces on the other side. In fact, if a real hypersurface $M$ is contact then $d\eta$ does not vanish identically at any point. We recall that $M$ is called a contact hypersurface if there exists an everywhere non-zero smooth function $\rho$ on $M$ such that $d\eta = 2\rho\omega$ where $\omega$ is the fundamental $2$-form defined by $\omega(X,Y) = g(\phi X, Y)$. On the contrary, a necessary condition for a real hypersurface to be co-symplectic or (nearly) Kenmutso is that the distribution $\cC$ be integrable or equivalently $d\eta \equiv 0$. In Proposition~\ref{not-contact} we show that for a Reeb flat real Hopf hypersurface $M \subset Q^m$, $d\eta$ does not vanish identically, at any point, but $M$ is not contact.

Our next result demonstrates another unique feature of the tube of radius $\pi/4$ around $\C P^k$ as a Reeb flat hypersurface:
\begin{theorem}\label{commuting ric}
Let $M$ be a Reeb flat real Hopf hypersurface in $Q^m$. Then $Ric \phi = \phi Ric$ if and only if $m$ is even, say $m = 2k$, and $M$ is a tube of radius $\frac{\pi}{4}$ around a $\C P^k \subset Q^m$.
\end{theorem}

In Proposition~\ref{parallel ric} we determine eigenvalues of shape operators of Reeb flat Hopf hypersurfaces with parallel Ricci tensor in $Q^m$. It is a correction to Theorem~1.2 of \cite{Suh15}. In that theorem the values $\pm\sqrt{3}$ are missing as principal curvatures of Reeb flat hypersurfaces. Besides correcting that mistake, our proof of Proposition~\ref{parallel ric} which uses Proposition~\ref{lambda} is different and shorter than that of Theorem~1.2 in \cite{Suh15}.  In the case of non-parallel Ricci tensor we prove the following non-existence theorem:
\begin{theorem}\label{killing ric}
There does not exist any Reeb flat real Hopf hypersurface in $Q^m$ with non-parallel Killing Ricci tensor.
\end{theorem}

Our last result, Proposition~\ref{S-is-Reeb parallel}, states that principal curvatures of a Reeb flat hypersurface is constant along the Reeb vector field and its shape operator is Reeb parallel. 
 
The paper is organized as follows. In Section~\ref{Sec:preliminaries} we provide preliminaries on $Q^m$ and real hypersurfaces in it. The proofs of the results are given in sections~\ref{Sec:commuting-jacobi} and  \ref{Sec:Reeb-flat}. In Section~\ref{Sec:commuting-jacobi} we get some general equations and results for real hypersurfaces in $Q^m$ with commuting structure Jacobi operator, although the section concludes with a classification result for the case of non-zero Reeb curvature. In Section~\ref{Sec:Reeb-flat} we focus on Reeb flat real Hopf hypersurfaces in $Q^m$.

\section{Preliminaries}\label{Sec:preliminaries}
In this section we provide preliminaries including what is needed on $Q^m$ and its real hypersurfaces. One can find the details in \cite{BerSuh13, BerSuh15, Smy67, Rec95}.
\subsection{The Complex Quadric}

We denote by $\C P^{m+1}$ the $(m+1)$-dimensional complex projective space equipped with the Fubbini-Study metric. Then, with  the homogeneous coordinates $z_1, \ldots, z_{m+1}$ on $\C P^{m+1}$, the complex hypersurface defined by the equation $z_1^2 + \cdots + z_{m+1}^2=0$ endowed with the induced K\"{a}hler structure $(g,J)$ from $\C P^{m+1}$ is a compact Hermitian symmetric space of rank two known as the complex quadric $Q^m$. The two-dimensional complex quadric $Q^2$ is isometric to $S^2 \times S^2$. In this paper we assume $m \geq 3$.

For a non-zero vector $z \in \C^{m+2}$ let $[z]$ denote the complex line spanned by $z$, that is $[z]=\C z$. Then, through the canonical identification the tangent space to $\C P^{m+1}$ at $[z]$ is given by $T_{[z]}\C P^{m+1} = \C^{m+2} \ominus [z]$ where $\C^{m+2} \ominus [z]$ denotes the orthogonal complement of the subspace $[z]$ in $\C^{m+2}$. It then follows that for $[z] \in Q^m$ we can identify $T_{[z]}Q^m$ with $\C^{m+2} \ominus ([z] \oplus [\bar{z}])= T_{[z]}\C P^{m+1} \ominus [\bar{z}]$. In particular, $\bar{z} \in \nu_{[z]}Q^m$ is a unit normal vector of $Q^m$ in $\C P^{m+1}$ at $[z] \in Q^m$.

For any unit vector $\zeta \in \nu_{[z]}Q^m$ the shape operator $A_\zeta$ of $Q^m$ is a conjugation on $T_{[z]}Q^m$, that is, $A_\zeta$ is an anti-linear involution. The operator $A_\zeta$ anti-commutes with the complex structure $J$, $A_\zeta J = -J A_\zeta$, and we have the decomposition 
\[
T_{[z]}Q^m = V(A_\zeta) \oplus JV(A_\zeta) \ ,
\]
where $V(A_\zeta)$ and $JV(A_\zeta)$ are, respectively, the $(+1)$ and $(-1)$-eigenspace of $A_\zeta$. For $\zeta = \bar{z}$ one gets $V(A_{\bar{z}})=\R^{m+2} \cap T_{[z]}Q^m$ and $JV(A_{\bar{z}})=i \R^{m+2} \cap T_{[z]}Q^m$.

The set $\A = \{A_\zeta : \zeta \in \nu Q^m \}$ of shape operators of $Q^m$ forms a parallel rank two subbundle of the endomorphism bundle $End(TQ^m)$, that is, for a certain $1$-form $q$ defined on $T_{[z]}Q^m$, $[z] \in Q^m$ we have $\bar{\nabla}A = JA \otimes q$, where $\bar{\nabla}$ denotes the induced connection from $\C P^{m+1}$ on $Q^m$. The subset $\A^0 = \{ \lambda A : \lambda \in S^1 \subset \C \}$ of $\A$ is a $S^1$-subbundle of conjugations and we have $V(\lambda A) = \lambda V(A)$.

From the Gauss equation for the complex hypersurface $Q^m \subset \C P^{m+1}$ one gets the following formula for the Riemannian curvature tensor $\bar{R}$ of $Q^m$ in terms of its Riemannian metric $g$, complex structure $J$ and a generic conjugation $A$ in $\A$:
\begin{align}\label{R-Q}
\bar{R}(X,Y)Z = & (X \wedge Y)Z + (JX \wedge JY)Z - 2g(JX , Y)JZ \\ \nonumber
 & + (AX \wedge AY)Z + (JAX \wedge JAY)Z 
\end{align}
for any $X$ and $Y$ tangent to $Q^m$, where 
\[
(U \wedge V)Z = g(V,Z)U - g(U,Z)V .
\] 

A non-zero vector $W \in T_{[z]}Q^m$ is said to be singular if it is tangent to more than one maximal flat in $Q^m$. For the complex quadric $Q^m$ a non-zero vector $W$ is singular if and only if either $W \in V(A)$ for some $A \in \A_{[z]}$ or $g(W,AW)=g(AW,JW)=0$ for some (and hence, any) $A \in \A_{[z]}$. In the first case $W$ is called a $\A$-principal singular vector and in the second case it is known as a $\A$-isotropic singular vector. 

For each unit tangent vector $W \in T_{[z]}Q^m$ there exist a conjugation $A_{[z]}$ and orthonormal vectors $X,Y \in V(A_{[z]})$ such that
\[
W = \cos(t)X + \sin(t)JY \ ,
\]
for some unique $t \in [0, \frac{\pi}{4}]$. The number $t$ is called the $\A$-angle or characteristic angle of $W$ and the vector $W$ is $\A$-principal, respectively $\A$-isotropic, if and only if its $\A$-angle is $0$, respectively $\frac{\pi}{4}$.

\subsection{Real Hypersurfaces in the complex quadric}
Now let $M$ be a real hypersurface of $Q^m$. Fix a (local) unit normal vector field $N$ on $M$. Then the unit vector field $\xi = - JN$ is tangent to $M$ and is known as the Reeb vector field of $M$. We denote its dual $1$-form with $\eta$ given by $\eta(X) = g(X,\xi)$ for any $X \in TM$. The tensor field $\phi $ on $M$ defined by $JX = \phi X + \eta(X)N$, with $X \in TM$, determines the tangential component of $J$ and is called the structure tensor field of $M$. The hypersurface $M$ together with the structure $(\phi , \xi, \eta, g)$ is an almost contact hypersurface of $Q^m$. The distribution $\cC = \ker \eta$ is the maximal complex subbundle of $TM$ and the tangent bundle $TM$ splits orthogonally into $TM = \cC \oplus \R \xi$. The tensor $\phi $ coincide with the complex structure $J$ restricted to $\cC$ and $\phi \xi = 0$. Given any $X$, $Y$ tangent to $M$, one has
\begin{equation}\label{eq0}
(\nabla_X\phi)Y = \eta(Y) SX - g(SX , Y)\xi , \quad \nabla_X\xi = \phi SX
\end{equation}
where $\nabla$ denotes the induced connection and $S$ stands for the shape operator of $M$ in $Q^m$.
   
Let $B$ and $\theta$ denote the tangent components of $A$ and $JA$, respectively, that is $AX = BX + g(AX , N)$ and 
\[
JAX = \theta X + g(JAX , N)N = \theta X + g(X , B\xi)N
\]
Then using \eqref{R-Q} and the Gauss equation for the hypersurface $M$ in $Q^m$ one gets the following formula for the curvature tensor of $M$:
\begin{equation}\label{R-M}
\begin{array}{rcl}
R(X,Y) & = & X \wedge Y + \phi X \wedge \phi Y - 2g(\phi X , Y)J + BX \wedge BY \\ && + \theta X \wedge \theta Y + SX \wedge SY
\end{array}
\end{equation}   

At each point $[z] \in M$ let $\Q_{[z]}$ denote the maximal $\A$-invariant subspace of $T_{[z]}M$ given by 
\[
\Q_{[z]} = \{ X \in TM : AX \in T_{[z]}M \ \ \text{for all} \ \ A \in \A_{[z]} \}
\]

The following lemma describes $\Q_{[z]}$.
\begin{lemma}[\cite{BerSuh13, Suh14}]\label{N-Type-characterization}
Let $M$ be a real hypersurface in $Q^m$ with unit normal vector field $N$. Then 
\begin{itemize}
\item[(i)]
The followings are equivalent:
\begin{itemize}
\item[(1)]
The normal vector $N_{[z]}$ is $\A$-principal;

\item[(2)]
$\Q_{[z]} = \cC_{[z]}$;

\item[(3)]
for some $A \in \A_{[z]}$, $AN_{[z]} \in \C N_{[z]}$;

\item[(4)]
for some $A \in \A_{[z]}$,  $AN = N$.
\end{itemize}

\item[(ii)]
If $N_{[z]}$ is not $\A$-principal, then there exist a conjugation $A \in \A_{[z]}$ and orthonormal vectors $X, Y \in V(A)$ such that $N_{[z]} = \cos(t)X + \sin(t)JY$ for some $t \in (0, \frac{\pi}{4}]$ and we have $\Q_{[z]} = \cC_{[z]} \ominus \C(JX + Y)$.

\item[(iii)]
If $N_{[z]}$ is $\A$-isotropic, then there is a conjugation $A \in \A_{[z]}$ such that $N_{[z]}$, $\xi_{[z]}$, $AN_{[z]}$ and $A\xi_{[z]}$ are pairwise orthogonal and $\Q_{[z]} = \cC_{[z]} \ominus (\R AN_{[z]} \oplus \R A\xi_{[z]})$.
\end{itemize}
\end{lemma}
Moreover, for a given point $[z] \in M$ if $N$ is $\A$-principal, respectively $\A$-isotropic, in a sufficiently small neighborhood $U \subset M$ around $[z]$ then a section $A \in \A^0$ can be chosen such that $A N = N$, respectively $A N \in \cC$, in $U$. Otherwise there exist a sufficiently small neighborhood $U$ of $[z]$ in $M$ and a section $A \in \A^0$ such that $g(AN,\xi) = 0$ in $U$. In what follows we always consider such a section when the normal vector is $\A$-principal or $\A$-isotropic. Note that if $N$ is not $\A$-principal ($\A$-isotropic) at $[z]$ then it is not $\A$-principal ($\A$-isotropic) in a neighborhood of $[z]$.

If any integral curve of the Reeb vector field $\xi$ is a geodesics of $M$, then $M$ is called a Hopf hypersurface. In this case $\xi$ is a principal curvature vector of $M$ everywhere and we have $S\xi = \alpha \xi$, where $\alpha = g(S\xi,\xi)$ is the Reeb curvature of $M$. It is shown that for a real Hopf hypersurface, $\alpha$ is constant if and only if $M$ is $\A$-principal or $\A$-isotropic, that is, the normal vector field $N$ is either $\A$-principal or $\A$-isotropic everywhere on $M$ (see \cite{Suh14}). 

Two groups $SO_{m+1} \subset SO_{m+2}$ and $U_{k+1} \subset SO_{m+2}$, with $m = 2k$, act by cohomogeneity one on $Q^m$, namely, the minimum codimension of their orbits is one. The action of $SO_{m+1}$ produces two singular orbits: one is the totally geodesic complex hypersurface $Q^{m-1} \subset Q^m$ and the other one is a totally geodesic $m$-dimensional sphere $S^m$. Tubes around a totally geodesic $Q^{m-1}$ in $Q^m$ are examples of $\A$-principal real hypersurfaces. These tubes admit contact structure and, in fact, it is proved that they are the only contact $\A$-principal real Hopf hypersurfaces in $Q^m$. 

The action of $U_{k+1}$ on even dimensional complex quadric $Q^m$ ($m = 2k$) also has two singular orbits and both of them are $\C P^k$ embedded in $Q^m$ as a complex submanifold. Tubes around a totally geodesic $\C P^k$ are of radius $0 < r < \pi/2$ and they are examples of $\A$-isotropic hypersurfaces in $Q^m$. Properties of these tubes are given in the following theorem.
\begin{theorem}[\cite{BerSuh13}] \label{tubes-CPk-properties}
Let $M$ be a tube of radius $0 < r < \pi/2$ around a totally geodesic $\C P^k$ in $Q^{2k}$, $k \geq 2$.  Then
\begin{itemize}
\item[($i$)]
$M$ is a Hopf hypersurface.

\item[($ii$)]
The tangent bundle $TM$ and the normal bundle $\nu M$ of $M$ consist of $\A$-isotropic singular tangent vectors of $Q^{2k}$.

\item[($iii$)]
The eigenvalues of the shape operator $S$ of $M$ and their multiplicities and corresponding eigenspaces are given in Table~\ref{table1}.

\item[($iv$)]
The shape operator $S$ and the structure tensor $\phi$ of $M$ commute with each other, that is, $S\phi = \phi S$.
\begin{table}\label{table1}
\caption{Eigenvalues of the shape operator of $M$}
\begin{tabular}{ccc}
\hline 
Eigenvalue & Eigenspace & Multiplicity \\ 
\hline \hline
$t \cot(2r)$ & $\R\xi$ & $1$ \\  
$0$ & $\cC \ominus \Q$ & $2$ \\ 
$-\tan(r)$ & $T\C P^k \ominus(\cC \ominus \Q)$ & $2k-2$ \\ 
$\cot(r)$ & $\nu\C P^k \ominus \C\nu M$ & $2k-2$\\
\hline 
\end{tabular} 
\end{table}

\item[($v$)]
The Reeb flow on $M$ is an isometric flow.
\end{itemize}
\end{theorem}

Moreover, it is also shown in \cite{BerSuh13} that properties ($iv$) and ($v$) are equivalent on a real Hopf hypersurface in $Q^m$ and tubes around a $\C P^k \subset Q^m$, with $m = 2k$ are the only hypersurfaces which have these properties.
  
\section{Real hypersurfaces with $\phi R_\xi = R_\xi\phi $}\label{Sec:commuting-jacobi}
In this section we study real hypersurfaces in $Q^m$ with the property that their structure Jacobi operator commutes with their structure tensor i.e., $\phi R_\xi = R_\xi\phi $.

The structure Jacobi operator for the hypersurface $M$ is defined by $R_\xi = R(.,\xi)\xi$. By \eqref{R-M} one has
\begin{align} \label{StrJacOpe}
R_\xi(X) & = X - \eta(X)\xi + (BX \wedge B\xi)\xi + g(AN,X)\phi A\xi + (SX \wedge S\xi)\xi \nonumber \\
& = X - \eta(X)\xi + (AX \wedge A\xi)\xi + g(AN,X)\phi A\xi + (SX \wedge S\xi)\xi
\end{align}
in which, since $X$ is tangent to $M$, one may replace $BX \wedge B\xi$ with $AX \wedge A\xi$. Moreover, from the Codazzi equation we get
\begin{align}\label{Codazzi1}
\begin{split}
(\nabla_XS)Y - (\nabla_YS)X =& \eta(X)\phi Y - \eta(Y)\phi X - 2g(\phi X,Y)\xi \\
& + g(AN,X)AY - g(AN,Y)AX \\
& + g(A\xi,X)JAY - g(A\xi,Y)JAX 
\end{split}
\end{align}
If we apply $\eta$ to both sides of \eqref{StrJacOpe} then 
\begin{align*}
\eta(R_\xi (X)) = & \eta(X) - \eta(X)\eta(\xi) + g(A\xi,\xi)\eta(AX) - g(AX,\xi)\eta(A\xi) \\
& - g(AN,X)\eta(AN) + g(S\xi,\xi)\eta(SX) - g(SX,\xi)\eta(S\xi)\\
=& 0
\end{align*}
for every $X \in TM$, from which and \eqref{StrJacOpe} one gets

\begin{equation}\label{phiR}
\begin{split}
\phi R_\xi(X)= &  JR_\xi(X) \\
 = & JX - \eta(X)J\xi + g(A\xi,\xi)JAX - g(A\xi,X)JA\xi \\
& - g(AN,X)JAN + g(S\xi,\xi)JSX - g(SX,\xi)JS\xi \\
= & \phi X + g(A\xi,\xi)JAX + g(A\xi,X)AN - g(AN,X)A\xi  \\
&+ g(S\xi,\xi)JSX - g(SX,\xi)JS\xi 
\end{split}
\end{equation}

On the other hand, we get $R_\xi\phi(X) = R_\xi(\phi X)$ using equation \eqref{StrJacOpe} as
\begin{equation}\label{Rphi}
\begin{split}
R_\xi\phi (X) = & \phi X + g(A\xi,\xi)A\phi X - g(AN,X)A\xi + g(A\xi,X)AN  \\
& + \eta(X)g(AN,N)AN + g(S\xi,\xi)S\phi X - g(S\phi X,\xi)S\xi \\
\end{split}
\end{equation}

Let the structure Jacobi operator $R_\xi$ commute with the structure tensor $\phi$. Then from equations \eqref{phiR} and \eqref{Rphi} we obtain
\begin{equation}\label{phiR-Rphi.0}
\begin{split}
0 = & \phi R_\xi(X) - R_\xi\phi (X) \\
= & g(A\xi,\xi)(JAX - A\phi X) + \eta(X)g(AN,N)AN + g(S\xi,\xi)(JSX - S\phi X) \\
& - g(SX,\xi)JS\xi + g(S\phi X,\xi)S\xi \\
= & g(A\xi,\xi)A\phi X + g(S\xi,\xi)(\phi S - S\phi )X - g(SX,\xi)\phi S\xi + g(S\phi X,\xi)S\xi 
\end{split}
\end{equation}

Taking inner product of \eqref{phiR-Rphi.0} with $\xi$ and using $g(A\phi X, \xi) = g(AN,X)$ implies 
\begin{equation}
0 = g(A\xi,\xi)g(AN,X)
\end{equation}
for every $X \in TM$. Thus, at each point either $g(A\xi,\xi)$ or $g(AN,X)$, for every $X \in TM$, vanishes. In the first case we have $g(AN,N)= - g(A\xi,\xi) = 0$ and in the second case it follows that $AN \in \C N$. So, by Lemma~\ref{N-Type-characterization} at each point of $M$ the normal vector $N$ is either $\A$-principal or $\A$-isotropic. Moreover, equation \eqref{phiR-Rphi.0} becomes
\begin{equation}\label{phiR-Rphi.1}
0 = \alpha(\phi S - S\phi )X - g(SX,\xi)\phi S\xi + g(S\phi X,\xi)S\xi 
\end{equation}
where $\alpha = g(S\xi,\xi)$ is the Reeb curvature function. Note that since we have not assumed the hypersurface $M$ to be Hopf, the relation $S\xi = \alpha \xi$ does not necessarily hold.

Given a point $[z] \in M$, assume that $N$ is $\A$-principal at $[z]$ and hence in a neighborhood of $[z]$ in $M$. By Lemma~\ref{N-Type-characterization} every conjugation $A \in \A$ leaves $\cC_{[z]}$ invariant. Let $A$ be a section of $\A$ that leaves $N$ invariant then the subspace $\cC_{[z]}$ splits orthogonally into 
\[
\cC_{[z]} = V(A|_\cC) \oplus JV(A|_\cC)
\]
where $V(A|_\cC)$ and $JV(A|_\cC)$ denote $+1$ and $-1$-eigenspaces of $A|_\cC$, respectively. One can see from $AJ = -JA$ that $\phi$ maps $V(A|_\cC)$ onto $JV(A|_\cC)$ and vice versa and according to \cite[Lemma~6]{Loo17},  $SY = 0$ for $Y \in JV(A|_\cC)$.

Now let $X = X_++X_- \in \cC_{[z]}$ be such that $SX = \lambda X$, for some $\lambda \in \R$, where  $X_+$ and $X_-$ are the components of $X$ with respect to the above decomposition for $\cC_{[z]}$. From \eqref{phiR-Rphi.1} we get
\begin{equation}\label{eq2}
0 = \lambda\alpha\phi X_+ +\lambda\alpha\phi X_- - \alpha\phi X_- + g(S\phi X_-,\xi)S\xi
\end{equation}
Apply $A$ to \eqref{eq2} to get:
\begin{equation}\label{eq3}
0 = -\lambda\alpha\phi X_+ +\lambda\alpha\phi X_- - \alpha\phi X_- + g(S\phi X_-,\xi)AS\xi
\end{equation}
Subtracting equations \eqref{eq2} and \eqref{eq3} gives us
\begin{equation}\label{eq4}
0 = 2\lambda\alpha\phi X_+ + g(S\phi X_-,\xi)S\xi - g(S\phi X_-,\xi)AS\xi
\end{equation}
which by taking inner product with $\xi$ concludes
\begin{align}\label{eq5}
\begin{split}
0 & = \alpha g(S\phi X_-,\xi) - g(S\phi X_-,\xi)g(AS\xi,\xi) \\
& = 2\alpha g(S\phi X_-,\xi)
\end{split}
\end{align}

If $\alpha$ is zero then we see from \eqref{phiR-Rphi.1} that $g(X,S\xi)\phi S\xi = g(X,\phi S\xi)S\xi$ and from that we get
\begin{equation}\label{eq13}
0 = g(X,S\xi)g(X,\phi S\xi)
\end{equation}
for every $X \in TM$. Since the  vectors $S\xi$ and $\phi S\xi$ are orthogonal, equation \eqref{eq13} holds for every $X \in TM$ only if $S\xi = 0$, which, in fact, means that $M$ is Hopf. Next, let $\alpha$ be non-zero, then from \eqref{eq5} we get $g(S\phi X_-,\xi) = 0$ and then by \eqref{eq4} one concludes $0 = \lambda\alpha\phi X_+$. By our assumption of $m \geq 3$ and given the fact that the real codimension of $\cC$ in $TM$ is one, it can be seen that either there exists some $X \in \cC_{[z]}$ such that $SX = \lambda X$ with $\lambda \neq 0$ or the shape operator $S$ is a scalar multiplication of the identity on $\cC_{[z]}$. The first cannot happen since then it follows from \eqref{eq5} that $\alpha = 0$. In the second case $M$ is Hopf and according to \cite{Loo17} it is (locally) a tube around a totally geodesic $Q^{m-1}$ in $Q^m$. As we have considered $A$ such that $AN = N$, $M$ turns out to have three distinct constant principal curvatures $\alpha$, $-\alpha$ and $0$ of respective multiplicities $1$, $m-1$ and $m-1$ and corresponding principal spaces $\R\xi$, $V(A|_\cC)$ and $JV(A|_\cC)$ (see \cite{BerSuh12}). Now, from \eqref{phiR-Rphi.1} we obtain:
\begin{equation}
(A\phi - \phi A)X = \alpha(S\phi - \phi S)X
\end{equation}
for every $X \in TM$. The above equation for $X \in V(A|_\cC)$ gives $2\phi X = \alpha^2 \phi X$ which clearly cannot hold if $X$ is a non-zero vector. So, all together, we have proved:

\begin{proposition}\label{N-isotropic}
Let $M$ be a real hypersurface in $Q^m$ such that $R_\xi\phi = \phi R_\xi$. Then, the unit normal vector field $N$ is $\A$-isotropic everywhere on $M$.
\end{proposition}

From now on we assume that the normal vector field $N$ is $\A$-isotropic everywhere on $M$.

\begin{proposition}\label{M-Hopf}
Let $M$ be a real hypersurface in $Q^m$ such that $R_\xi\phi = \phi R_\xi$. Then, $M$ is a Hopf hypersurface.
\end{proposition}
\begin{proof}
Given $Y \in TM$, if we take covariant derivative of \eqref{phiR-Rphi.1} in $Y$ direction and put $X=\xi$ we get 
\begin{align*}
Y(\alpha)\phi S\xi = g((\nabla_YS)\xi,\xi)\phi S\xi  
\end{align*}

On the other hand, we have
\begin{align*}
Y(\alpha) & = g(\nabla_YS\xi,\xi) + g(S\xi,\nabla_Y\xi) \\
& = g((\nabla_YS)\xi,\xi) + 2g(S\phi SY,\xi)
\end{align*}

Comparing these two equations yields $\phi S\xi = 0$ or $g(SY,\phi S\xi) = - g(S\phi SY,\xi) = 0$. If $\phi S\xi = 0$ at a point, then $S\xi \in \R\xi$ which means $M$ is Hopf at that point. For the other case, let $\phi S\xi$ be non-zero, then $g(SY,\phi S\xi)$ vanishes for every $Y \in TM$. By taking covariant derivative of $\phi S\xi$ in the direction of $\xi$ in $Q^m$ we get
\begin{align*}
\bar{\nabla}_\xi \phi S\xi & = \nabla_\xi \phi S\xi + g(S\xi, \phi S\xi)N \\
& = (\nabla_\xi \phi )(S\xi) + \phi \nabla_\xi S\xi \\
& = \alpha S\xi - g(S\xi, S\xi)\xi + \phi \nabla_\xi S\xi
\end{align*}
Applying $J$ to the above equation gives
\begin{align*}
\bar{\nabla}_\xi J\phi S\xi & = - \bar{\nabla}_\xi S\xi + \bar{\nabla}_\xi\eta(\phi S\xi)N \\
& = \alpha JS\xi - g((S\xi,S\xi)N - \nabla_\xi S\xi + \eta(\phi \nabla_\xi S\xi)N \\
& = \alpha JS\xi - g((S\xi,S\xi)N - \nabla_\xi S\xi
\end{align*}
while according to Weingarten formula $\bar{\nabla}_\xi S\xi = \nabla_\xi S\xi - g(S\xi,S\xi)N$. So, it follows that $\alpha JS\xi = 0$ which by our assumption that $\phi  S\xi$ and hence  $JS\xi$ being non-zero, concludes $\alpha = 0$. Putting $\alpha = 0$ into \eqref{phiR-Rphi.1} implies 
\[
 g(SX,\xi)\phi S\xi = g(S\phi X,\xi)S\xi
\]
which for $X = S\xi$ becomes
\[
g(S\xi,S\xi)\phi S\xi = g(\phi S\xi,S\xi)S\xi
\]

The right hand side of the above equation is zero by our assumption and thus we obtain $g(S\xi,S\xi) = 0$, and consequently, $S\xi = 0$. So, again the hypersurface $M$ turns out to be Hopf. 
\end{proof}

Since we have proved that $M$ is Hopf, we can use the relation $S\xi = \alpha\xi$ hereafter. It is proved in \cite{BerSuh13} that a $\A$-isotropic real Hopf hypersurface in $Q^m$ has constant Reeb curvature. Hence, Proposition~\ref{M-Hopf} together with Proposition~\ref{N-isotropic} implies

\begin{proposition}\label{alpha-constant}
Let $M$ be a real hypersurface in $Q^m$ such that $R_\xi\phi = \phi R_\xi$. Then, the Reeb curvature function $\alpha$ is constant.
\end{proposition} 

The typical examples of real Hopf hypersurfaces in $Q^m$ with $\A$-isotropic normal vector field are tubes around a totally geodesic $\C P^k$, with $m = 2k$. We conclude this section with the proof of our main result of the section, Theorem~\ref{alpha-non-zero}, which determines these tubes as the only hypersurfaces with non-zero Reeb curvature satisfying the relation $R_\xi\phi = \phi R_\xi$. 

\begin{proof}[Proof of Theorem~\ref{alpha-non-zero}]
Note that equations \eqref{phiR} and \eqref{Rphi} now become
\[
\phi R_\xi(X)= \phi X + g(A\xi,X)AN - g(AN,X)A\xi + \alpha \phi SX  
\]
and
\[
R_\xi\phi (X) = \phi X + g(A\xi,X)AN - g(AN,X)A\xi + \alpha S\phi X 
\]
and thus we have
\begin{equation}\label{phiR-Rphi.2}
\phi R_\xi(X) - R_\xi\phi (X) = \alpha(\phi S - S\phi )X 
\end{equation}

By Theorem~\ref{tubes-CPk-properties}, tubes around $\C P^k \subset Q^m$, with $m = 2k$, have isometric Reeb flow, or equivalently, have commuting shape operator ($S\phi = \phi S$). Moreover, these tubes are the only real Hopf hypersurfaces in $Q^m$ with isometric Reeb flow. On the other hand, according to \eqref{phiR-Rphi.2}, if $\alpha$ is non-zero the condition $R_\xi\phi = \phi R_\xi$ on $M$ is equivalent to having commuting shape operator. 
\end{proof}

\section{Reeb flat real Hopf hypersurfaces}\label{Sec:Reeb-flat}
In this section we investigate real Hopf hypersurfaces in $Q^m$ with $\alpha = 0$ known as Reeb flat hypersurfaces. 

As we mentioned earlier for a real Hopf hypersurface in $Q^m$ if $\alpha$ is constant then the normal vector field $N$ is either $\A$-principal or $\A$-isotropic everywhere on $M$. Our first result in this section rules out that $N$ can be $\A$-principal in the Reeb flat case:

\begin{proposition}\label{isotropic pro}
Let $M$ be a Reeb flat real Hopf hypersurface in $Q^m$. Then, the normal vector field $N$ is $\A$-isotropic everywhere on $M$.
\end{proposition}

\begin{proof}

From \eqref{Codazzi1} we get 
\begin{align*}
\frac{1}{2}\eta((\nabla_{X}S)Y-(\nabla_{Y}S)X)=& -g(\phi X,Y)+g(X,AN)g(Y,A\xi)\\
& -g(Y,AN)g(X,A\xi)
\end{align*}
in which we used $g(JX,A\xi)=-g(X,JA\xi)=g(X,AN)$. On the other hand, using $\nabla_X\xi = \phi SX$ and the assumption $S\xi=0$ one gets
\begin{align*}
\eta((\nabla_{X}S)Y)&=g(\nabla_{X}SY,\xi)-g(\nabla_{X}Y,S\xi)\\
&=Xg(SY,\xi)-g(SY,\nabla_{X}\xi)\\
& =g(\phi SY,SX)\\
& =g(S\phi S Y,X)
\end{align*}
and similarly $g((\nabla_{Y}S)X,\xi)=g(S\phi S Y,X)$. So
\begin{align*}
\eta((\nabla_{X}S)Y-(\nabla_{Y}S)X)=2g(S\phi S Y,X)=-2g(Y,S\phi SX)
\end{align*}

Comparing the above two expressions for $\eta((\nabla_{X}S)Y-(\nabla_{Y}S)X)$ we obtain  
\begin{align}\label{eq6}
\begin{split}
g(S \phi SX,Y)=& g(\phi X,Y)-g(X,AN)g(Y,A\xi)\\ & +g(Y,AN)g(X,A\xi)
\end{split}
\end{align}
which for $X=\xi$ yields
\begin{equation}\label{eq11}
g(\xi,AN)g(Y,A\xi)=g(Y,AN)g(\xi,A\xi)
\end{equation}

Let $A$ be a conjugation for which $N=\cos(t) Z_{1}+\sin(t) JZ_{2}$ where $Z_1,Z_2 \in V(A)$ are orthonormal and $0\leq t \leq \pi/4$. Then 
\begin{align*}
AN & =\cos(t) Z_{1}-\sin(t) JZ_{2}, \\
\xi & =\sin(t) Z_{2}-\cos(t) JZ_{1},\\ 
A\xi & =\sin(t)Z_{2}+\cos(t)JZ_{1}.
\end{align*}
So, $g(\xi,AN)=0$ and from \eqref{eq11} we get $g(Y,AN)g(\xi,A\xi)=0$. If $g(\xi,A\xi)=0$ then $N$ is $\A$-isotropic. Otherwise $g(Y,AN)=0$ for every $Y$, which means that $AN \in \C N$ and thus by Lemma~\ref{N-Type-characterization} $N$ is $\A$-principal and we may assume $AN = N$.

If $N$ is $\A$-principal then as in the proof of Proposition~\ref{N-isotropic} one can see that $J(A|_{\cC}) \subset \ker S$. On the other hand, if $SX = 0$ for any $X \in \cC$ (i.e., $\eta(X) = 0$) then by \eqref{eq6} we get $g(\phi X , Y) = g(S \phi S X , Y) = 0$, for every $Y$, meaning that $\phi X = 0$. Hence, we have $0 = \phi^2 X = -X - \eta(X)\xi = -X$ which implies $J(A|_{\cC}) = \{0\}$ and this is a contradiction since we have assumed that $m \geq 3$. 
\end{proof}

For any non-zero real eigenvalue $\lambda$ of the shape operator $S$, let $\Q_{[z]}(\lambda)$ denote its corresponding eigenspace of $S$ in $T_{[z]}M$. One can see that $\Q_{[z]}(\lambda) \subset \Q_{[z]}$, and as a result $\Q_{[z]}$ decomposes into $\Q_{[z]}= \bigoplus_{\lambda \neq 0}\Q_{[z]}(\lambda)$. We have the following proposition on eigenvalues of the shape operator of $M$.

\begin{proposition}\label{lambda}
If $0 \neq \lambda$ is an eigenvalue of $S$ then so is $\frac{1}{\lambda}$ and for every $X \in \Q(\lambda)$ we have $\phi X \in \Q(\frac{1}{\lambda})$. In particular, $dim \Q(\lambda) = dim \Q(\frac{1}{\lambda})$.
\end{proposition}
\begin{proof}
Equation \eqref{eq6} can be written as 
\begin{align*}
S\phi SX=\phi X-g(X,AN)A\xi+g(X,A\xi)AN
\end{align*}

According to Proposition~\ref{isotropic pro}, $N$ is $\A$-isotropic, and by Lemma~5.1 in \cite{PerJeoKoSuh18} we have $SA\xi=SAN=0$. By applying $S$ to both sides of the above equation one gets  
\begin{equation}\label{eq8}
S^{2}\phi SX=S\phi X\qquad \forall X\in T_{[z]}M
\end{equation}

It follows form \eqref{eq8} that if $SX=0$ then $S\phi X=0$, as well. Also, using
\[
 g(\phi X,Y)=-\frac{1}{\lambda}g(SX,\phi Y)=-\frac{1}{\lambda}g(X,S\phi Y)
\]
one can see that if $X\in \Q(\lambda)$ and $SY=0$ then $g(\phi X,Y)=0$. Let $X\in \Q(\lambda)$, $\lambda \neq 0$, and choose $Y\in \Q(\mu)$ for some $\mu$ such that $g(\phi X,Y)\neq 0$. Note that $\mu$ has to be non-zero. Then we have 
\begin{align*}
g(S^{2}\phi SX,Y)=g(\phi SX,S^{2}Y)=\lambda \mu^{2}g(\phi X,Y)
\end{align*}

We can also use \eqref{eq8} to calculate
\begin{align*}
g(S^2\phi SX,Y)=g(\phi X,SY)=\mu g(\phi X,Y)
\end{align*}

Comparing the above equations we get $\lambda \mu^{2}=\mu$ and therefore $\mu=\frac{1}{\lambda}$. So, $\frac{1}{\lambda}$ is also an eigenvalue of $S$.

Let us write $\phi X=\sum f_{i} Y_{i}$ with respect to the decomposition $\Q_{[z]}= \bigoplus_{\lambda \neq 0}\Q_{[z]}(\lambda)$ where $Y_{i}\in \Q(\mu_{i})$. As we have just shown if $f_{i}\neq 0$, for any $i$, then it turns out that $\mu_{i}=\frac{1}{\lambda}$ and thus we have $S\phi X=\sum f_{i}\frac{1}{\lambda}Y_{i}=\frac{1}{\lambda}\phi X$, hence, $\phi X\in \Q(\frac{1}{\lambda})$, as we claimed. 
 \end{proof}

So, given a basis $X_{1},\ldots, X_{k}$ for $\Q(\lambda)$, $\phi X_{1}, \ldots, \phi X_{k}$ is a basis for $\Q(\frac{1}{\lambda})$. Note also that $S\phi SX=\lambda\frac{1}{\lambda}\phi X=\phi X$ for $X \in \Q(\lambda)$ and therefore 
\begin{align}\label{eq9}
S\phi SX=\phi X
\end{align}
for all $X\in \bigoplus_{\lambda\neq0}\Q(\lambda)$.

The proof of the previous proposition, in particular, implies that the subspaces $\Q^\perp$ and $\Q(\lambda) \oplus \Q(\frac{1}{\lambda})$ are invariant under both $S$ and $\phi$.

Proposition~\ref{lambda} imposes a rather striking condition on tubes around $\C P^k$ for being flat, so that as stated in Theorem~\ref{commuting shape operator} only one of them turns out to be flat. Here is the proof of the theorem.

\begin{proof}[Proof of Theorem~\ref{commuting shape operator}]
We see from Proposition~\ref{lambda} that $S\phi X=\frac{1}{\lambda}\phi X$ and $\phi SX=\lambda \phi X$ for $X\in\Q(\lambda)$. So, the only possible non-zero values for $\lambda$ in order that the shape operator $S$ to be commuting are $+1$ and $-1$. Therefore, if $S$ is commuting it has three distinct eigenvalues $0$, $-1$, and $1$ with respective multiplicities $3$, $m-4/2$, $m-4/2$ and eigenspaces $\spann\{\xi, A\xi, AN \}$, $\Q(1)$ and $\Q(-1)$, where $\Q(1)$ and $\Q(-1)$ are $\phi$-invariant and $m$ turns out to be even. So, it follows from Theorem~\ref{tubes-CPk-properties} that $M$ is a tube of radius $\frac{\pi}{4}$ around a $\C P^k \subset Q^m$ with $m=2k$.

\end{proof}

The shape operator of $M$ is said to be of Codazzi type if $(\nabla_XS)Y=(\nabla_YS)X$ for all tangent vector fields $X$ and $Y$ on $M$.

In the next proposition we give some features of Reeb flat hypersurfaces which show that they are different from some known families of hypersurfaces such as contact, co-symplectic and (nearly) Kenmutso hypersurfaces in K\"{a}hler manifolds. 

\begin{proposition}\label{not-contact}
Let $M$ be a Reeb flat real Hopf hypersurface in $Q^m$. Then
\begin{itemize}
\item[($i$)]
the tensor $\eta((\nabla_{X}S)Y$ is not symmetric with respect to $X$ and $Y$; in particular, the shape operator of $M$ is not of Codazzi type,

\item[($ii$)]
$M$ is not contact,

\item[($iii$)]
$d\eta$ does not vanish identically at any point.
\end{itemize}
\end{proposition}
\begin{proof}
Let $M$ be a Reeb flat real Hopf hypersurface in $Q^m$. Using~\ref{eq0} and equation \eqref{eq9} we get
\[
\eta((\nabla_{X}S)Y)=g(\nabla_{X}SY,\xi)=g(\phi SY,SX)=g(S\phi SY,X)=g(\phi Y,X)
\]
and then
\[
\eta((\nabla_{X}S)Y-(\nabla_{Y}S)X)=2g(\phi Y,X)
\]
for $X,Y\in \Q$. The above equation shows that the tensor $\eta((\nabla_{X}S)Y)$ is not symmetric with respect to $X$ and $Y$ since the right hand side of the equation does not vanish for $Y = \phi X$. In particular, $M$ is not of Codazzi type. 

We also have
\[
\eta(\nabla_{X}Y)=X\eta(Y)-g(Y,\nabla_{X}\xi)=g(\phi Y,SX)
\]
from which we get
\begin{align*}
\eta([X,Y])&=g(\phi Y,SX)-g(\phi X,SY)\\
&=g(S\phi Y,X)+g(X,\phi SY)\\
&=g((S\phi +\phi S)Y,X)
\end{align*}
and then with $X\in \Q(\lambda)$ and $Y=\phi X$ one gets
\begin{align*}
\eta([X,\phi X])=(\lambda+\frac{1}{\lambda})\|X\|^{2}
\end{align*}
If $M$ is contact, according to \cite{BerSuh15} there exists a non-zero constant $K$ such that $(S\phi +\phi S)=K\phi $. For $X\in \Q(\lambda)$, the constant $K$ must clearly be $\lambda+\frac{1}{\lambda}$. While if we put $X=A\xi$ then, since $\phi A\xi=-AJ\xi = AN$, it follows that $SA\xi =SAN=0$ and $(S\phi +\phi S)A\xi=0$ thus $K$ should be zero, a contradiction. Therefore, $M$ cannot be a contact hypersurface.

To see that $d\eta$ does not vanish identically at any point, note that if $d\eta \equiv 0$ then $\eta([X,Y])=0$ for all $X,Y\in \R\eta^{\perp}$ and by the above equation $\lambda^{2}+1=0$ holds, which is not possible. 
\end{proof}

In what follows we give some characterizations for Reeb flat hypersurfaces in terms of Ricci tensor. 

By contracting equation \eqref{R-M} one gets
\begin{align*}
Ric(X)=&(2m-1)X-X-\phi ^{2}X-2\phi ^{2}X-g(AN,N)AX\\
&+g(AX,N)AN-g(JAN,N)JAX-X\\
&+g(JAX,N)JAN-\tr(S)SX-S^{2}X\\
=&(2m-1)X-3\eta(X)\xi -g(AN,N)AX+g(AX,N)AN\\
&-g(JAN,N)JAX+g(JAX,N)JAN+\tr(S)SX-S^{2}X
\end{align*}
which for Reeb flat hypersurfaces, where $N$ is $\A$-isotropic, becomes 
\begin{align*}
Ric(X)= &(2m-1)X-3\eta(X)\xi +g(X,AN)AN\\
&-g(AX,\xi)JAN+\tr(S)SX-S^{2}X
\end{align*}
and then we have
\begin{align*}
Ric(\xi)=&(2m-1)\xi -3\xi +g(\xi ,AN)AN \\
&+g(A\xi , \xi)JAN-\tr(S)S\xi -S^{2}\xi\\
=&(2m-4)\xi \\
Ric(A\xi)= &(2m-1)A\xi-3\eta(A\xi)\xi+g(A\xi ,AN)AN \\ 
&+g(A^{2}\xi,\xi)A\xi-\tr(S)SA\xi -S^{2}A\xi \\
=& (2m)A\xi \\
Ric(AN)= &(2m-1)AN-3\eta(AN)\xi +g(AN,AN)AN\\
& +g(A^{2}N,\xi)-\tr(S)SAN -S^{2}AN \\
=&(2m)AN
\end{align*}
and also for $X\in \Q(\lambda)$ 
\[
Ric(X)=(2m-1+\lambda.\tr(S)-\lambda^{2})X
\]

So, for a Reeb flat real Hopf hypersurface in $Q^m$, the eigenspaces $\Q^\perp$ and $\Q(\lambda)$ are eigenspaces of the Ricci operator, as well, and we have
\begin{align*}
Ric(\xi) & = (2m-4)\xi, \\
Ric(A\xi) & = (2m)A\xi, \\
Ric(AN) & = (2m)AN, \\
Ric(X) & = (2m-1+\lambda.\tr S - \lambda^2)X, \quad \text{ for all } X \in \Q(\lambda) .
\end{align*}

In the following proposition we determine non-zero principal curvatures and their principal curvature spaces of Reeb flat hypersurfaces with parallel Ricci tensor. 
   
\begin{proposition}\label{parallel ric}
Let $M$ be a Reeb flat real Hopf hypersurface in $Q^m$. If the Ricci tensor of $M$ is parallel then its shape operator has four distinct non-zero eigenvalues $\pm\sqrt{3}$, $\pm\frac{\sqrt{3}}{3}$ and their corresponding eigenspaces are of the same dimension.
\end{proposition}
\begin{proof}
Let $X\in \Q(\lambda)$ then
\[
Ric(\nabla_{X}\xi)=Ric(\phi SX)=Ric(\lambda\phi X)=\left(2m-1+\tr(S)\frac{1}{\lambda}-\frac{1}{\lambda^2}\right)\lambda\phi X
\]
and
\[
\nabla_{X}Ric(\xi)=\nabla_{X}(2m-4)\xi=(2m-4)\nabla_{X}\xi=(2m-4)\phi SX
\]

Since the Ricci tensor is assumed to be parallel we get $2m-4 = 2m-1+\tr(S)\frac{1}{\lambda}-\frac{1}{\lambda^2}$ which can be simplified as $3\lambda^{2}+\tr(S)\lambda-1=0$. The latter equation shows that $\lambda$ takes two distinct non-zero values $\lambda_1$ and $\lambda_2$, which according to Proposition~\ref{lambda}, their corresponding eigenspaces have the same dimension, say $k$. So, we have
\[
\tr(S)=\lambda_{1}+\lambda_{2}+\frac{1}{\lambda_{1}}+\frac{1}{\lambda_{2}}=\frac{2}{3} k \tr(S)
\]
which implies $\tr(S)=0$ and therefore we get $\lambda_{1}=\frac{\sqrt{3}}{3}$ and $\lambda_{2}=-\frac{\sqrt{3}}{3}$. Thus, the Ricci tensor of $M$ has four distinct non-zero eigenvalues $\pm\sqrt{3}$, $\pm\frac{\sqrt{3}}{3}$.
\end{proof}

\begin{proof}[Proof of Theorem~\ref{commuting ric}]
Let $X\in\Q(\lambda)$ then
\[
Ric(\phi X)=\left(2m-1+\frac{1}{\lambda}.\tr(S)-\frac{1}{\lambda^2}\right)\phi X
\]
and
\[
\phi (Ric(X))=\left(2m-1+\lambda.\tr(S)-\lambda^{2}\right)\phi X
\]
So, the Ricci tensor is commuting if and only if $\frac{1}{\lambda}.\tr(S)-\frac{1}{\lambda^2}=\lambda.\tr(S)-\lambda^{2}$ or, equivalently, $\lambda^{4}-\lambda^{3}.\tr(S)+\lambda.\tr(S)-1=0$. In this case, the possible values of $\lambda$ are
\[
\lambda =1,-1,\frac{\tr(S)}{2}+\sqrt{\tr(S)^{2}-4}/2,\frac{\tr(S)}{2}-\sqrt{\tr(S)^{2}-4}/2
\]
As in the proof of Proposition~\ref{parallel ric} one can conclude that $\tr(S) = 0$ and therefore the values taken by $\lambda$ are $+1$ and $-1$. Also note that $\dim \Q(1)=\dim\Q(-1)$ and $S\phi =\phi S$. The result then follows from Theorem~ \ref{commuting ric}.
\end{proof}

We recall that the Ricci tensor of $M$ is said to be Killing if $(\nabla_{\xi}Ric)X+(\nabla_{X}Ric)\xi=0$ for every tangent vector field $X$ on $M$. Here we give the proof of Theorem~\ref{killing ric} which states the Ricci tensor of a Reeb flat hypersurface is parallel if it is Killing.

\begin{proof}[Proof of Theorem~\ref{killing ric}]
From \eqref{Codazzi1} we get
\begin{align*}
-\phi X & =(\nabla_{X}S)\xi-(\nabla_{\xi}S)X\\
& = \nabla_{X}S\xi -S\nabla_{X}\xi-(\nabla_{\xi}S)X \\
& =-S\phi SX-(\nabla_{\xi}S)X
\end{align*}
for $X\in \Q$, which according to \eqref{eq9} yields
 \begin{equation}\label{eq10}
 (\nabla_{\xi}S)X=0
 \end{equation}

Next, we show that $\nabla_{\xi}X \in \Q$. Hence, since $M$ is Hopf and $S\xi=0$ one has 
\[
g(\nabla_{\xi}X,\xi)=\xi g(X,\xi)-g(X,\nabla_{\xi}\xi)=0
\]
and also 
 \begin{align*}
 g(\nabla_{\xi}X,A\xi)&=\xi g(X,A\xi)-g(X,\nabla_{\xi}A\xi)\\
 &=-g(X,\bar{\nabla}_{\xi}A\xi)\\
 &=-g(X,q(\xi)AN)\\
 &=0
 \end{align*}
where we used $(\bar{\nabla}_{\xi}A)\xi =q(\xi)AN$ and $(\bar{\nabla}_{\xi}A)\xi =\bar{\nabla}_{\xi}A\xi -A(\bar{\nabla}_{\xi}\xi)=\bar{\nabla}_{\xi}A\xi$. Similarly, for $g(\nabla_{\xi}X,AN)$ we have 
\begin{align*}
 g(\nabla_{\xi}X,AN)&=\xi g(X,AN)-g(X,\nabla_{\xi}AN)\\
 &=-g(X,\bar{\nabla}_{\xi}AN)\\
 &=g(X,q(\xi)A\xi)\\
 &=0
 \end{align*} 
in which $(\bar{\nabla}_{\xi}A)N =-q(\xi)A\xi$ and $(\bar{\nabla}_{\xi}A)N =\bar{\nabla_{\xi}}AN -A(\bar{\nabla}_{\xi}N)=\bar{\nabla}_{\xi}AN$ have been used. Therefore $\nabla_{\xi}X$ is perpendicular to $\Q^\perp = \spann\{\xi,A\xi,AN\}$ and thus $\nabla_{\xi}X\in\Q$. Moreover, we show that $\nabla_{\xi}X\in\Q(\lambda)$ if $X$ is a unit vector field in $\Q(\lambda)$. In fact, let $Y\in \Q(\mu)$ be orthogonal to $X$ so that $g(\nabla_{\xi}X,Y)\neq 0$, then
 \[
 g(\nabla_{\xi}X,Y)=\frac{1}{\mu}g(\nabla_{\xi}X,SY)=\frac{1}{\mu}g(S\nabla_{\xi}X,Y)=g(\nabla_{\xi}SX,Y)
 \]
On the other hand
 \begin{align*}
 g(\nabla_{\xi}X,Y)& =\frac{1}{\mu}(\xi g(SX,Y)-g(SX,\nabla_{\xi}Y) \\
 & =-\frac{\lambda}{\mu}g(X,\nabla_{\xi}Y)\\
 &=-\frac{\lambda}{\mu}(\xi g(X,Y)-g(Y,\nabla_{\xi}X) \\
 &=-\frac{\lambda}{\mu}g(Y,\nabla_{\xi}X)
 \end{align*}
So, it follows that $\lambda=\mu$ and $\nabla_{\xi}X\in\Q(\lambda)$.
 
We define a function $f$ on $M$ by $f=2m-1+\lambda.\tr(S)-\lambda^{2}$. Then we have $Ric(\nabla_{\xi}X)=f\nabla_{\xi}X$ and $\nabla_{\xi}Ric X = \nabla_{\xi}(f X)=\xi(f)X+f\nabla_{\xi}X$ from which one gets  $(\nabla_{\xi}Ric)X=\xi(f)X$. We also have
 \begin{align*}
 (\nabla_{X}Ric)\xi &=\nabla_{X}Ric(X)-Ric(\nabla_{X}\xi)\\
 &=(2m-4)\nabla_{X}\xi-\lambda. Ric(\phi SX)\\
 &=(2m-4)\phi SX-\lambda. Ric(\phi X)\\
 & =\lambda(-3-\lambda.\tr(S)+\lambda^{2})\phi X
 \end{align*}
If the Ricci tensor of $M$ is Killing, that is, $(\nabla_{\xi}Ric)X+(\nabla_{X}Ric)\xi=0$, then from the above equations we get
 \[
 \xi(f)X=-\lambda(-3-\lambda.\tr(S)+\lambda^{2})\phi X
 \]
which shows that the function $f$ is constant in the direction of $\xi$ and $\lambda(-3-\lambda.\tr(S)+\lambda^{2})=0$. Then, as in the proof of Proposition~\ref{parallel ric}, one can see that $\tr(S)=0$ and the shape operator $S$ has four distinct non-zero eigenvalues $\pm\sqrt{3}$, $\pm\frac{\sqrt{3}}{3}$ with corresponding eigenspaces of the same dimension. In particular, it turns out that in this case $(\nabla_{\xi}Ric)X=(\nabla_{X}Ric)\xi=0$ and the Ricci tensor is parallel.
\end{proof} 

Finally, we have the following proposition on principal curvatures and shape operator of Reeb flat hypersurfaces in $Q^m$.
\begin{proposition}\label{S-is-Reeb parallel}
Let $M$ be a Reeb flat real Hopf hypersurface in $Q^m$. Then its principal curvatures are constant along $\xi$. Moreover, the shape operator is Reeb parallel, namely, $\nabla_\xi S =0$.
\end{proposition}

\begin{proof}
Let $\lambda$ be a principal curvature of $M$. Using \eqref{eq10} we have $\nabla_\xi SX = S \nabla_\xi X$ for all $X \in \Q$ . So, if $X \in \Q(\lambda)$, then as in the proof of Theorem~\ref{killing ric} we get $\nabla_\xi\lambda X = \lambda \nabla_\xi X$, that is, $\xi(\lambda) + \lambda \nabla_\xi X = \lambda \nabla_\xi X$ which gives $\xi(\lambda) = 0$ and thus $\lambda$ is constant along $\xi$. 

To prove that $S$ is Reeb parallel, let $Y \in \Q^\perp$ then $SY = 0$. Moreover, for any $X \in \Q(\lambda)$ we have $\lambda g(\nabla_\xi Y , X) = -\lambda g(Y , \nabla_\xi X)$. Since $\nabla_\xi X \in \Q(\lambda)$ it follows that  
\[
\lambda g(Y , \nabla_\xi X) = g(Y , S\nabla_\xi X) = g(SY , \nabla_\xi X) = 0
\]
So, $\nabla_\xi Y \in \Q^\perp$ and $S(\nabla_\xi Y) = 0$. Consequently, $(\nabla_\xi S)Y = 0$ for all $Y \in \Q^\perp$ which together with \eqref{eq10} implies that the shape operator is Reeb parallel. 
\end{proof}

\section*{Acknowledgment}
This work was conducted during a postdoctoral fellow visit of the third author at Tarbiat Modares University with additional support from Iran national science foundation via grant no. 95012382.

The authors would like to thank Prof. Young Jin Suh for his comments on the first version of the paper.

\providecommand{\href}[2]{#2}
\bibliographystyle{amsplain}

\end{document}